\documentclass[reqno, 12pt]{amsart}
\usepackage[mathscr]{eucal}
\usepackage{amscd}
\usepackage{color}
\usepackage{amsfonts}
\usepackage{amsmath, amsthm, amssymb}
\usepackage{latexsym}
 \newtheorem{thm}{Theorem}[section]
 \newtheorem{cor}[thm]{Corollary}
 \newtheorem{lem}[thm]{Lemma}
 \newtheorem{prop}[thm]{Proposition}
 \newtheorem{defn}[thm]{Definition}
 \newtheorem{rem}[thm]{Remark}
 
 \numberwithin{equation}{section}

\def\R{{\Bbb R}}
\def\a{{\alpha}}
\def\la{{\lambda}}
\def\pl{{\partial}}
\def\R{{\mathbb R}}
\def\N{{\Bbb N}}

\def\no{{\nonumber}}

\def\bge{\begin{eqnarray}}
\def\bgee{\begin{eqnarray*}}
\def\ege{\end{eqnarray}}
\def\egee{\end{eqnarray*}}
\newcommand{\dd}{{\mathrm{d}}}
\newcommand{\ee}{{\mathrm{e}}}

\newcommand{\calF}{{\mathcal{F}}}

\newcommand{\dif}[2]{\frac{\dd #1}{\dd #2}}

\newcommand{\I}[1]{{\, \dd #1}}

\newcommand{\Ir}{{\, \dd r}}

\newcommand{\Ix}{{\, \dd x}}

\newcommand{\vr}{v_r}
\newcommand{\vrr}{v_{rr}}

\newcommand{\vt}{v_t}

\newcommand{\zrr}{z_{rr}}

\newcommand{\zt}{z_t}

\newcommand{\jr}{J_r}
\newcommand{\jrr}{J_{rr}}
\newcommand{\jt}{J_t}

\newcommand{\jti}{{\tilde{J}}}
\newcommand{\jtir}{\jti_r}
\newcommand{\jtirr}{\jti_{rr}}
\newcommand{\jtit}{\jti_t}

\def\bea{\bge}
\def\eea{\ege}

\begin{document}
\title[Quenching of  nonlocal parabolic MEMS equation]
{On the quenching of a nonlocal parabolic problem arising in electrostatic MEMS control}

\author{Nikos I. Kavallaris}

\address{Department of Mathematics, University of Chester, Thornton Science Park,
Pool Lane, Ince, Chester  CH2 4NU, UK}
\email{n.kavallaris@chester.ac.uk}

\author{Andrew A. Lacey}
\address{Maxwell Institute for Mathematical Sciences \& 
School of Mathematical and Computer Sciences,
Heriot-Watt University,
Riccarton, Edinburgh, EH14 4AS, UK
}

\email{A.A.Lacey@hw.ac.uk}

\author{Christos V. Nikolopoulos}
\address{Department of Mathematics, University of Aegean,
Gr-83200 Karlovassi, Samos, Greece
}

\email{cnikolo@aegean.gr}



\date{\today}

\begin{abstract}

We consider a nonlocal parabolic model for a micro-electro-mechanical system.
Specifically, for a radially symmetric problem with monotonic initial
data, it is shown that the solution quenches, so that touchdown occurs
in the device, in a situation where there is no steady state. It is also
shown that quenching occurs at a single point and a bound on the approach
to touchdown is obtained. Numerical simulations illustrating the results
are given.

\end{abstract}

\maketitle
\section{Introduction}
\setcounter{equation}{0}
The main purpose of the current work is to investigate
a singular mathematical behaviour, called {\it quenching},
of the solutions of the following non-local parabolic problem:
\bea
& & u_t-\Delta u ={\displaystyle\frac{\lambda}{\left( 1-u \right)^2
\left(1+\a \int_{\Omega}
\frac{1}{1-u} \I x\right)^2}}
\quad \mbox{in}\quad Q_T:=\Omega\times (0,T), \label{eqn:1}\\
& & u=0 \quad \mbox{on}\quad \partial \Omega\times (0,T), \label{bc}\\
& & u(x,0)=u_0(x), \quad x\in\Omega,\label{ic}
\eea
where $\lambda,\a$ are positive constants, $T>0$,
$\Omega\subset \R^N$ is a bounded domain with smooth boundary $\partial\Omega$,
and $u_0$ is a continuous function in $\overline{\Omega}$ such that
$0\le u_0<1$.

The motivation for studying problem \eqref{eqn:1}-\eqref{ic} is that
it arises as a mathematical model which describes operation of
some electrostatic actuated micro-electro-mechanical systems (MEMS).
The term ``MEMS" more precisely refers to precision devices
which combine mechanical processes
with electrical circuits. In particular, electrostatic actuation is
a popular application of MEMS. MEMS devices range in size from
millimetres down to microns, and involve precision mechanical
components that can be constructed using semiconductor
manufacturing technologies. Various electrostatic actuated MEMS
have been developed and used in a wide variety of devices applied as sensors
and have fluid-mechanical, optical, radio frequency (RF),
data-storage, and biotechnology applications.
Examples of microdevices of this kind include microphones,
temperature sensors, RF switches, resonators, accelerometers,
micromirrors, micropumps, microvalves, data-storage devices etc.,
\cite{EGG10,JAP-DHB02,y}.

The principal part of such a electrostatic actuated
MEMS device usually consists of an
elastic plate suspended above a rigid ground plate.
Typically the elastic plate (or membrane) is held
fixed at two ends while the other two edges remain free to move.
An alternative configuration could entail the plate or membrane
being held fixed around its entire edge. When a potential difference
$V_d$ is applied between the membrane and the plate,
the membrane deflects towards the ground plate. Under
the realistic assumption that the width of the gap, between the membrane
and the bottom plate, is small compared to the
device length, and when the configuration of
the two parallel plates is connected in series with
a fixed voltage source and a fixed capacitor, then
the deformation of the elastic membrane $u$ is
described by the dimensionless equation \eqref{eqn:1},
see \cite{PT01, JAP-DHB02}. In \eqref{eqn:1}
$$
\la=\frac{V_d^2 L^2 \varepsilon_0}{2\mathcal{T}l^2},
$$
$\mathcal{T}$ is the tension in the membrane,
$L$ the characteristic length (diameter) of the domain
$\Omega,\;$ $l$ the characteristic width of the gap
between the membrane and the fixed ground plate (electrode),
and $\varepsilon_0$ the permittivity of free space.
The integral in \eqref{eqn:1} arises from the fact that
the device is embedded
in an electrical circuit with a capacitor of fixed capacitance.
The parameter $\alpha$ denotes the ratio of this fixed capacitance
to a reference capacitance of the device. Without loss of generality,
we may assume that $\alpha=1$.
(The limiting case $\alpha=0$ corresponds to the configuration
where there is no capacitor in the circuit.)
If the edges of the membrane are kept fixed then
Dirichlet boundary conditions of the form
(\ref{bc}) are imposed. It is usually supposed that
the elastic membrane is initially at rest, so that
$u(x,0)\equiv 0.$ However, in this work, we consider
more general non-negative initial conditions
$u(x,0)=u_0(x)\geq 0$. For a more detailed derivation of \eqref{eqn:1}
see \cite{PT01, JAP-DHB02}.

From many experiments it is clear that the applied voltage $V_d$
controls the operation of the MEMS device. It is observed that when $V_d$
exceeds a critical threshold $V_{cr}$, called the
{\it pull-in voltage}, then the phenomenon of
{\it touch-down} (or {\it pull-in instability} as it is also known
in MEMS literature) occurs when the elastic membrane touches
the rigid ground plate.
For the mathematical problem \eqref{eqn:1}-\eqref{ic},
this means that there is some critical value
$\la_{cr}$ of the parameter
$\la$ above which singular behaviour should be anticipated.
Focusing on the nonlinear term of problem \eqref{eqn:1},
one can notice that such singular behaviour is possible only when
$u$ takes the value $1$, a phenomenon  known in the mathematical literature
as {\it quenching}, see also Section~\ref{fq}.
From the point of view of applications it is important
to determine whether quenching occurs and, if it does,
to clarify when, how and where it might happen.

\

When $\alpha=0$ we obtain the following local (standard) parabolic problem,
\bgee\label{local}
& & u_t-\Delta u= \frac{\la }{(1-u)^2} \quad \mbox{in} \quad Q_T \, , \\
& & u=0 \quad \mbox{on}\quad \partial \Omega\times (0,T) \, , \\
& & u(x,0)=u_0(x) \quad \mbox{for} \quad x\in\Omega \, ,
\egee
whose quenching behaviour has been extensively studied in the papers
\cite{FMPS07, GG08, YG-ZP-MJW06,  G08,  KMS08}.

\

The quenching behaviour of the solutions of \eqref{eqn:1}-\eqref{ic}
has been also studied in \cite{GHW08,gk12, H11} but questions regarding : \\
$(i)$ occurrence of quenching for $\lambda>\lambda^*$
where $\la^*$ is defined by \eqref{crl}; \\
$(ii)$ determination of the quenching rate; \\
$(iii)$ establishment of the form of the quenching set; \\
were left open.
These questions are addressed in the current work for
radially symmetric problems with the extra assumption that the
initial data decreases with distance from the centre of the domain.

It is worth mentioning that the problem
\eqref{eqn:1}-\eqref{ic} shares some common features with a
non-local problem which exhibits blow-up and models control of mass, and
which was investigated in \cite{hy95, QS07}. However, the methods
we use in this paper to examine finite-time quenching are rather different
from the those of \cite{hy95, QS07}.
\

The structure of the paper is as follows. In Section~\ref{ss}
a brief study of corresponding steady-state problem is presented;
this will be used for proving the occurrence of finite-time quenching
in Section~\ref{fq}. Section~\ref{spq} is devoted to determination
of the form of the quenching set
and a bound on the quenching rate. Finally Section \ref{nr}
presents some numerical results confirming some of the analytical results
obtained in the preceding sections.


\section{Steady-State Problem}\label{ss}

\setcounter{equation}{0}


In this section a brief study of the steady-state problem corresponding
to \eqref{eqn:1}-\eqref{ic} is provided. This problem has the form
\begin{equation}\label{nss}
\Delta w+ \frac{\la}{(1-w)^2 \Big(1+\int_{\Omega} \frac{\dd x}{1-w}
\Big)^2}=0,\;\;x\in \Omega,\quad w=0,\;\; x\;\in \pl \Omega,
\end{equation}
where we always have $0\leq w <1$ in $\overline{\Omega}$
for a (classical) solution of \eqref{nss}.

Let
\bge\label{crl}
\la^*:=\sup\{\la>0:\;\; \mbox{problem \eqref{nss} admits a classical solution}\},
\ege
then $\la^*<\infty$ for any dimension $N\geq 1,$ see \cite{GHW08, gk12}.
For more on the structure of the solution set of  \eqref{nss} see \cite{gk12}.

For the purposes of the current work we will need the notion
of a weak solution of problem \eqref{nss}. In particular we define
the following form of weak solution for \eqref{nss}.

\begin{defn}\label{def0}
A function $w\in H_0^1(\Omega)$ is called a {\bf weak finite-energy solution}
of \eqref{nss} if there exists a sequence $\{w_j\}_{j=1}^{\infty}
\in C^2(\Omega)\cap C_0(\Omega)$ satisfying as $j\to \infty$
\bge
&& w_j\rightharpoonup w \quad\mbox{weakly in}
\quad H_0^1(\Omega)\label{ws1} \, , \\
&& w_j\to w\quad \mbox{a.e.,}\label{ws2} \\
&& \frac{1}{(1-w_j)^2}\to \frac{1}{(1-w)^2}\quad\mbox{in}
\quad L^1(\Omega)\label{ws3} \, , \\
&& \frac{1}{(1-w_j)}\to \frac{1}{(1-w)}\quad\mbox{in}\quad L^1(\Omega)\label{ws4}
\ege
and
\bge
\Delta w_j+ \frac{\la}{(1-w_j)^2 \Big(1+\int_{\Omega}
\frac{\dd x}{1-w_j} \Big)^2}\to 0 \quad\mbox{in}\quad L^2(\Omega)\label{ws5}.
\ege
\end{defn}

It follows that any weak finite-energy solution of \eqref{nss} also satisfies
\bgee
-\int_{\Omega} \nabla \phi \cdot \nabla w \I x +
\la\frac{\int_{\Omega} \frac{\phi}{(1-w)^2}\I x}{\Big(1+\int_{\Omega}
\frac{\dd x}{1-w}\Big)^2} =0\quad \mbox{for all}\quad \phi \in H_0^1(\Omega),
\egee
i.e. it is  a weak $H^1_0(\Omega)-$solution of \eqref{nss} as well,
see also \cite{ZH12}.

Set
\bgee
\widehat{\la}:=\sup\left\{\la>0:\;\; \mbox{problem \eqref{nss}
admits a weak finite-energy solution}\right\}.
\egee

We now restrict our discussion to radially symmetric problems, so
that we may take $\Omega=B_1=B_1(0)=\{x\in \R^N: ||x||_2 <1\}$ with
solutions which are decreasing in $r = ||x||_2$.
The relation between $\la^*$ and $\widehat{\la}$ is then
provided by the following:

\begin{prop}\label{sp1}
For radially symmetric problems, with radially decreasing solutions,
the suprema of the spectra of the classical and
weak problems are identical: $\la^*=\widehat{\la}.$
\end{prop}
\begin{proof}

Since any classical solution of \eqref{nss} is also
a weak finite-energy solution, $\la^*\leq \widehat{\la}$.

On the other hand, we can take $\la_1$ arbitrarily close to
$\widehat{\la}$ so that there is a weak finite-energy solution
$w_1$ for $\la = \la_1$. Since $w_1$ is decreasing with $0 \le w \le 1$,
either $w < 1$ for $0 < r \le 1$ or there is some $s>0$ such that
$w=1$ for $0 \le r < s$. In the latter case $\int_\Omega (1-w)^{-1} \Ix$
becomes infinite so that $w$ is then a weak finite-energy solution
of $\Delta w = 0$ satisfying $0 \le w \le 1$, as well as the
boundary condition $w=0$, giving $w \equiv 0$. We must then have
$w_1(r) < 1$ for $r > 0$ and it follows that $w_1$ is regular
for $r > 0$.

For $N=1$, simple integration now gives that the solution is
classical. For $N \ge 2$, following now \cite{gk12} (see also
\cite{JL73}), the (classical) problem can be solved in
$r>0$ to find that there is precisely one limiting value of $\la$,
say $\la_*$, for which $w(0) = 1$ and $w$ is then
a weak finite-energy solution but not classical. Depending
upon the value of $N$, $\la_* < \la^*$ or $\la_* = \la^*$.
In either case, $\la^* = \widehat{\la}$.
\end{proof}


\section{Finite-Time Quenching}\label{fq}

\setcounter{equation}{0}


Local-in-time existence of a solution to problem \eqref{eqn:1}-\eqref{ic}
is established in Corollary 2.4 of \cite{gk12} by constructing
a proper lower-upper pair or solutions. Moreover the solution $u$
exists as long as $u<1$ and it ceases to exist once $u$ reaches 1.
In particular we have:
\begin{defn}
The solution $u(x,t)$ of problem \eqref{eqn:1}-\eqref{ic} quenches at
some point $x^*\in \Omega$
in finite time $0<T_q<\infty$ if there exist sequences
$\{x_n\}_{n=1}^{\infty}\in \Omega $ and
$\{t_n\}_{n=1}^{\infty}\in (0,\infty)$ with $x_n\to x^*$ and $t_n\to T_q$
as $n\to\infty$ such
that $u(x_n,t_n)\to 1-$ as $n\to\infty$. In the case where
$T_q=\infty$ we say that $u(x,t)$ quenches in infinite time at $x^*.$

The set
\bgee
\mathcal{Q} = \Big\{x^*\in\overline{\Omega} &\Big|&
	\mbox{there exists a sequence $(x_k,t_k)_{k\in\N} \subset
	\Omega\times (0,T_q)$ such that } \\[0mm]
	& & x_k\to x^*, t_k\to T_q \mbox{ and } u(x_k,t_k)\to 1
	\mbox{ as } k\to\infty \Big\}
\egee
is called the quenching set.
\end{defn}
Some finite-time quenching results for large values of the parameter
$\lambda$ have been proved in \cite{GHW08, H11}, while the authors
in \cite{gk12} proved a quenching result for initial data
$u_0$ quite close to $1.$

In the current section we are working towards the improvement
of the preceding quenching results under some circumstances.
Before we proceed with the proof of our quenching results
we will present some auxiliary lemmata.

First of all we point out that problem \eqref{eqn:1}-\eqref{ic}
admits an energy functional of the form
\bge\label{eng}
E[u](t)\equiv E(t)=\frac{1}{2}\int_{\Omega} |\nabla u|^2 \I x +
\frac{\lambda}{\left(1+\int_{\Omega}(1-u)^{-1} \I x\right)}>0,
\ege
which decreases with respect to time for a solution of problem
\eqref{eqn:1}-\eqref{ic}. More precisely,
\bge\label{deng}
\dif Et = - \int_{\Omega} u_t^2(x,t) \I x < 0 \quad \mbox{for} \quad t>0\;,
\ege
or, equivalently,
\bge\label{deng1}
E(t) = \frac{1}{2}\int_{\Omega} |\nabla u|^2 \I x
+ \frac{\lambda}{\left(1+\int_{\Omega}(1-u)^{-1} \I x\right)}\leq E(0)
= E_0<\infty.
\ege
A key estimate for  proving our quenching result is given by
the following:
\begin{lem}\label{est1}
Let $u$ be a global-in-time solution of problem \eqref{eqn:1}-\eqref{ic}.
Then there is a sequence $\{t_j\}_{j=1}^{\infty}\uparrow \infty$
as $j\to \infty$ such that
\bge\label{tdc}
\la\int_{\Omega} u_j(1-u_j)^{-2} \I x\leq C_1 H^2(u_j),
\ege
for some positive constant $C_1,$ where $u_j=u(\cdot, t_j)$ and
\begin{equation} \label{def:H}
H(u_j) := 1+\int_{\Omega}(1-u_j)^{-1} \I x > 0.
\end{equation}
\begin{proof}
Suppose that the problem \eqref{eqn:1}-\eqref{ic} has a global-in-time
solution $u(x,t)=u(x,t;\la)$. Then,
multiplying equation (\ref{eqn:1}) by $u$ and integrating over $\Omega$,
we derive
\bge\label{q1}
\int_{\Omega} u\,u_t \Ix & = & \int_{\Omega} u{\displaystyle\left[\Delta u
+\frac{\la(1-u)^{-2}}{\Big(1 + \int_{\Omega}(1
- u)^{-1} \Ix \Big)^2}\right]} \I x\nonumber \\
& = & - \int_{\Omega}|\nabla u|^2 \I x + \frac{\la\int_{\Omega}
u(1-u)^{-2} \I x}{\Big(1 + \int_{\Omega}(1-u)^{-1} \I x \Big)^2} \nonumber \\
& = & - 2 E(t) + \frac{2\la}{1+\int_{\Omega}(1-u)^{-1} \I x}
+ \frac{\la\int_{\Omega} u(1-u)^{-2} \I x}{\Big(1
+ \int_{\Omega}(1-u)^{-1} \I x \Big)^2} \, ,
\ege
using also integration by parts and relation (\ref{eng}).

By virtue of H{\"o}lder's inequality and \eqref{deng1},
\eqref{q1} implies
\bge\label{q2}
\la\int_{\Omega} u(1-u)^{-2} \Ix & = & 2 E(t)H^2(u)-2\la H(u)
+ H^2(u)\,\int_{\Omega} u\,u_t \Ix \nonumber \\
& \leq & 2 E_0 H^2(u)+||u(\cdot,t)||_2\,||u_t(\cdot,t)||_2\,H^2(u) \nonumber \\
& \leq & 2 E_0 H^2(u)+|\Omega|^{1/2}\,||u_t(\cdot,t)||_2\,H^2(u).
\ege
On the other hand, the energy dissipation formula \eqref{deng} reads
\bgee
0\leq \int_{\tau}^t\int_{\Omega} u_t^2(x,s) \Ix \I s = E(\tau) - E(t)
\egee
and thus from \eqref{deng1} we deduce that
\bge\label{disp}
\int_{\tau}^\infty\int_{\Omega} u_t^2(x,s) \Ix \I s \leq C < \infty,
\ege
where the constant $C$ is independent of $\tau$.

Now \eqref{disp} yields the existence of a sequence
$\{t_j\}_{j=1}^{\infty}\uparrow \infty$ such that
\bge\label{tdc0}
\,||u_t(\cdot,t_j)||^2_2 = \int_{\Omega} u_t^2(x,t_j) \Ix \to
0 \quad \mbox{as} \quad t_j\to \infty,
\ege
and thus by virtue of \eqref{q2}
\bge\label{tdc1}
\la\int_{\Omega} u_j(1-u_j)^{-2} \Ix \leq C_1 H^2(u_j),
\ege
for some $C_1>0$.
\end{proof}
\end{lem}
The next step is to provide another key estimate for $H(u)$
which will allow us not only to prove finite-time quenching
but also to characterize the form of the quenching set.
However, such an estimate according to our method can
be only obtained for the radial symmetric case,
i.e. when $\Omega=B_1$. Then problem \eqref{eqn:1}-\eqref{ic}
in $N$ dimensions is written as
\bea\label{eqn:r}
& & u_t-\Delta_r u =F(r,t),
\quad (r,t)\in(0,1)\times (0,T), \label{eqn:r1}\\
& & u_r(0,t)=u(1,t)=0, \quad t\in (0,T), \label{rbc}\\
& & 0\leq u(r,0)=u_0(r)<1, \quad0<r<1,\label{ric}
\eea
where $\Delta_r u:=u_{rr}+(N-1)r^{-1}u_r$ for $N\geq 1$ and
\bge\label{ps}
F(r,t) = \la k(t)(1-u(r,t))^{-2},
\ege
for
\bgee
k(t) = \left( 1 + N \omega_N\int_{0}^1 r^{N-1}\,(1-u(r,t))^{-1} \Ir \right)^{-2},
\egee
where $\omega_N=|B_1|=\pi^{N/2}/\Gamma(N/2)$ stands for the volume
of the $N$-dimensional unit sphere $B_1(0)$ in $\R^N$ and
$\Gamma$ is the gamma function.

Condition $u_r(0,t)=0,$ for $N\geq 1$ is imposed
to guarantee the regularity of the solution $u.$
If we consider radial decreasing initial data $u_0(r),$ i.e.,
$u_0'(r)\leq 0,$ then it is a standard result that
the monotonicity property is inherited
by $u$ so that $u_r (r,t)\le 0$ for $r>0$ and $t>0$.

For the sake of simplicity  we  obtain the desired estimate for
$H(1-v)$ where $v$ is defined as $v:=1-u.$ Then $v\to 0+$ if $u\to 1-$.
Moreover $v$ satisfies
\bge
&&\vt - \vrr - (N-1)r^{-1}\vr = -fv^{-2}, \quad (r,t)\in(0,1)\times (0,T),
\label{eqn:basic}\\
&& v_r(0,t)=0,\; v(1,t)=1, \quad t\in (0,T),\label{bc:u:centre}\\
&& 0< v(r,0)=v_0(r)\leq 1, \quad0<r<1,
\label{ic:u}
\ege
where
\begin{equation}
f = f(t) = \frac\lambda {(1 + N \omega_N \int_0^1 r^{N-1} v^{-1} \I r )^2} .
\label{eqn:fdef}
\end{equation}
Then we have the following:
\begin{lem}\label{est2}
Consider symmetric and radial increasing initial data $v_0(r)$.
Then for any $k>2/3$ there exists a positive constant $C(k)$ such that
\bge\label{ps0a}
1-u(r,t)\geq C(k) r^k\quad\mbox{for}\quad (r,t)\in (0,1)\times (0,T_{max})
\ege
where $T_{max}$ is the maximum existence time of solution $u.$

Furthermore, there exists $C_2$ uniform in
$\lambda$ and independent of time $t$ such that
\bge\label{ps0}
H(u) = H(1-v) \leq C_2\quad\mbox{for any}\quad 0<t<T_{max}.
\ege
\end{lem}
\begin{proof}
Fixing some $a$, $1<a<2$, there are some $t_1 > 0$ and $\epsilon_1 > 0$
such that
\begin{equation}
\vr > \epsilon_1 r v^{-a} \, \mbox{ at $t = t_1$ for $0<r<1$.}
\label{bd:1}
\end{equation}
We define
\begin{equation}
z = r^{N-1} \vr
\label{def:w}
\end{equation}
and it is then easy to check, \cite{FM}, by differentiating (\ref{eqn:basic}), that
\begin{equation}
\zt - \zrr + (N-1) r^{-1}z_r = 2r^{N-1}fv^{-3} \vr \, .
\label{eq:w}
\end{equation}
We define
\begin{equation}
J = z - \epsilon r^{N}v^{-a}
\label{def:J}
\end{equation}
where
\[
0 < \epsilon < \epsilon_1 \, .
\]
Then
\begin{equation}
\jt = \zt + a\epsilon r^{N}v^{-a-1}\vt \, ,
\label{eq:Jt}
\end{equation}
\begin{equation}
\jr = z_r + a\epsilon r^{N}v^{-a-1}\vr - N\epsilon r^{N-1}v^{-a} ,
\label{eq:Jr}
\end{equation}
and
\begin{equation}
\jrr = \zrr + a\epsilon r^{N}v^{-a-1}\vrr + 2Na\epsilon r^{N-1}v^{-a-1}\vr
- a(a+1)\epsilon r^{N}v^{-a-2}\vr^2 - N(N-1)\epsilon r^{N-2}v^{-a} .
\label{eq:Jrr}
\end{equation}
We define a function $G(\epsilon)$ by
\begin{equation}
G(\epsilon) = \frac{\epsilon^{\frac 2{a+1}}}
{(\epsilon^{\frac 1{a+1}} + \frac{ N \omega_N}{Na + N - 2} (a+1)^{\frac a{a+1}}
2^{\frac 1{a+1}} )^2} .
\label{def:F}
\end{equation}
Our choice of $\epsilon$ is then, more precisely, given by
\begin{equation}
0 < \epsilon < \min\{\epsilon_1,\epsilon_2\} \, ,
\label{ineq:epsilon}
\end{equation}
where $\epsilon_2 > 0$ is chosen to satisfy
\begin{equation}
\epsilon_2 < \sup \left\{ \epsilon : \epsilon \le \min \left\{ \frac 1N ,
\left( \frac{2-a}{2a} \right) \right\} \lambda G(\epsilon) \right\} \, ;
\label{ineq:epsilon2}
\end{equation}
a small $\epsilon_2$ satisfying (\ref{ineq:epsilon2}) can be found
since $G(\epsilon)$ is of order $\epsilon^{\frac 2{a+1}} \gg \epsilon$
for $\epsilon$ small (recall that $a>1$).
Then
\begin{equation}
J > 0 \, \mbox{ for $0<r \le 1$ at $t=t_1$.}
\label{ic:J}
\end{equation}
As long as $J>0$,
\begin{equation}
z > \epsilon r^N v ^{-a} \quad \Rightarrow \quad \vr > \epsilon r v ^{-a}
\quad \Rightarrow \quad v > \left( \frac{(a+1)\epsilon}2 \right)^{\frac 1{a+1}}
r^{\frac 2{a+1}}
\label{estimateonv}
\end{equation}
which will lead to \eqref{ps0a}.

Then
\[
\int_0^1 r^{N-1} v^{-1} \I r < \left( \frac 2{(a+1)\epsilon}
\right)^{\frac 1{a+1}}
\int_0^1 r^{\frac{Na+N-2}{a+1} - 1} \I r =
\left( \frac 2{(a+1)\epsilon} \right)^{\frac 1{a+1}}
\left( \frac{a+1}{Na+N-2} \right)
\]
so
\begin{equation}
f(t) = \frac\lambda {(1 + N \omega_N \int_0^1 r^{N-1} v^{-1} \I r )^2}
> \lambda G(\epsilon) \, .
\label{ineq:feps}
\end{equation}
In particular, $f(t) > \lambda G(\epsilon)$ in a neighbourhood
of $t=t_1$.

We suppose for a contradiction that
\begin{equation}
\mbox{there is some $t_2 \in (t_1,T_{max})$ such
that $f(t_2) = \la G(\epsilon)$ with $f(t) > \la G(\epsilon)$ for $t_1 \le t < t_2$.}
\label{cont:t2}
\end{equation}

Now
\begin{equation}
J = 0 \, \mbox{ on $r=0$.}
\label{bc:J0}
\end{equation}
On the boundary $r=1$ we have
\[
J = \vr - \epsilon \, \mbox{ and then }
\]
\[
\jr = \vrr + (N-1)\vr + a\epsilon\vr - N\epsilon = (f - (N-1)\vr) + (N-1)\vr
+ a\epsilon\vr - N\epsilon
= f + a\epsilon\vr - N\epsilon
\]
\begin{equation}
\mbox{so both} \quad
\jr - a\epsilon J = f + a\epsilon^2 - N\epsilon > f - N\epsilon \,
\mbox{ on $r=1$}
\label{bc:J2}
\end{equation}
\begin{equation}
\mbox{and, since $\vr \ge 0$ on $r=1$,} \quad
\jr \ge f - N\epsilon \,
\mbox{ on $r=1$.}
\label{bc:J3}
\end{equation}
Provided that
\begin{equation}
\epsilon < f/N \, ,
\label{bound:eps1}
\end{equation}
either (\ref{bc:J2}) or (\ref{bc:J3}) gives a positive boundary condition
on $r=1$.

Now
\[
\jt - \jrr + (N-1) r^{-1}\jr = 2r^{N-1}fv^{-3} \vr + (a\epsilon r^N v^{-a-1})
(2(N-1)r^{-1}\vr - fv^{-2}) - 2Na\epsilon r^{N-1} v^{-a-1} \vr
\]
\[
+ a(a+1)\epsilon r^{N} v^{-a-2} \vr^2 + N(N-1)\epsilon r^{N-2}v^{-a}
- N(N-1)\epsilon r^{N-2}v^{-a}
\]
\[
> (2r^{N-1}fv^{-3} + 2(N-1)a\epsilon r^{N-1} v^{-a-1} -
2Na\epsilon r^{N-1}v^{-a-1})\vr - a\epsilon fr^N v^{-a-3}
\]
\[
= 2(fv^{-3} - a\epsilon v^{-a-1})z - a\epsilon fr^N v^{-a-3}
\]
\[
= 2(fv^{-3} - a\epsilon v^{-a-1})J + 2(fv^{-3} -
a\epsilon v^{-a-1})\epsilon r^N v^{-a} - a\epsilon fr^N v^{-a-3}
\]
\[
= 2(fv^{-3} - a\epsilon v^{-a-1})J + \epsilon (2-a)f r^N v^{-a-3}
- 2a\epsilon^2 r^N v^{-2a-1} \, .
\]
Thus
\begin{equation}
\jt - \jrr + (N-1) r^{-1}\jr > 2(fv^{-3} - a\epsilon v^{-a-1})J \, ,
\label{ineq:jde}
\end{equation}
as long as
\begin{equation}
\epsilon <  \frac{(2-a)}{2a}f \, .
\label{bound:eps2}
\end{equation}

Following the standard arguments for the maximum principle,
we now show that $J>0$ for $0<r\le 1$, $t_1 \le t \le t_2$.

In $0<r\le1$, $t_1 \le t \le t_2$, because $v>0$, the coefficient
of $J$ in (\ref{ineq:jde}) is bounded. We can then define a new variable
$\jti = \ee^{-D_1 t}J$ which then satisfies boundary condition
(\ref{bc:J0}), boundary inequality (\ref{bc:J2}) and
\begin{equation}
\jtit - \jtirr + (N-1) r^{-1}\jtir > - D_2 \jti \, ,
\label{ineq:jtide}
\end{equation}
where $D_1$ and $D_2$ are positive constants. Should $\jti$ be
non-positive somewhere (with $r>0$), it must take a non-positive
minimum at some $(r_3,t_3)$ with $0<r_3 \le 1$ and $t_1 < t_3 \le t_2$.

For $r_3 = 1$, (\ref{bc:J3}) gives $\jtir > 0$ on $r=1$,
leading to
a contradiction, so the supposed minimum must have $0<r_3<1$, where
$\jtit \le 0$, $\jtir = 0$ and $\jtirr \ge 0$. With $\jti \le 0$
and $D_2 > 0$, (\ref{ineq:jtide}) gives another contradiction.
Hence both $\jti$ and $J$ remain positive in $r>0$ for $t_1 \le t \le t_2$.

\

This now gives that (\ref{ineq:feps}) holds at $t=t_2$, contradicting
the assumption (\ref{cont:t2}). Thus, as long as the solution exists,
$f(t) > \lambda G(\epsilon)$ for $t \ge t_1$.

It follows that $J>0$, (\ref{ps0a}) holds and
\begin{equation}
\int_0^1 r^{N-1} v^{-1} \I r < \frac 1{Na+N-2} (a+1)^{\frac a{a+1}}
\left( \frac 2{\epsilon} \right)^{\frac 1{a+1}}
\label{bound:int}
\end{equation}
for all $t \ge t_1$, if (\ref{eqn:basic})-(\ref{ic:u}) have a global
solution, and up to and including the quenching time,
if the solution quenches. Thanks to the definition of $H(u)$ (\ref{bound:int})
implies the desired estimate.
\end{proof}
\begin{rem}
An estimate similar to \eqref{ps0a} has been also obtained
in \cite{GHW08} but only for the one-dimensional case.
Furthermore, here we also prove that the exponent $2/3$
is optimal for the validity of \eqref{ps0a}, see Section~\ref{spq}.
\end{rem}


\subsection{Quenching for $\la>\la^*$}

We now can combine Lemma \ref{est1} and Lemma \ref{est2} to derive
the following quenching result:
\begin{thm}\label{que1}
Consider symmetric and radial decreasing initial data $u_0(r)$.
Then for any $\la>\la^*$ the solution
of problem \eqref{eqn:r1}-\eqref{ric} quenches in finite time $T_q<\infty.$
\end{thm}
\begin{proof}
Let $\la>\la^*$ and assume that problem \eqref{eqn:r1}-\eqref{ric}
has a global-in-time solution. Then \eqref{tdc} in conjunction with
\eqref{ps0} yields
\bge\label{q3}
\la N \omega_N\int_{0}^1 r^{N-1} u_j(1-u_j)^{-2} \Ir \leq C_3 \, ,
\quad \mbox{for any} \quad t>0,
\ege
where the constant $C_3$ is independent of $j$.

From this and \eqref{ps0} we have
\bge\label{q4}
N \omega_N \int_{0}^1 \frac{r^{N-1} \Ir}{(1-u_j)^2} & = &
N \omega_N \int_{0}^1 \frac{r^{N-1} \Ir}{(1-u_j)} +
N \omega_N \int_{0}^1 \frac{r^{N-1} u_j \Ir}{(1-u_j)^2}\nonumber \\
& \leq & (C_2 - 1) + C_3/\la := C_4
\ege
where $C_4$ is independent of $j$.

From the energy dissipation formula \eqref{deng1}
we also have
\bge\label{ps4}
||\nabla u_j||_{L^2(B_1)}^2\leq C_5<\infty,
\ege
with constant $C_5$ being independent of $j$ as well.
Passing to a subsequence, if necessary, relation \eqref{ps4}
implies the existence of a function $w$ such that
\bge
&&u_j\rightharpoonup w\quad\mbox{in}\quad H^1_0(B_1),\label{q8}\\
&&u_j\to w\quad\mbox{a.e.}\,\quad\mbox{in}\quad B_1.\label{q9}
\ege
For $N\geq 2$ by virtue of \eqref{ps0a} we directly derive that $1/(1-u_j)^2$ is uniformly integrable and since
$1/(1-u_j)^2\to 1/(1-w)^2$, a.e. in $B_1$, due to \eqref{q9},
we deduce
\bge\label{q10}
\frac{1}{(1-u_j)^2}\to \frac{1}{(1-w)^2} \quad \mbox{as}
\quad j\to \infty\quad\mbox{in}\quad L^1(B_1),
\ege
applying the dominated convergence theorem. Similarly we also derive
\bge\label{ps2}
H(u_j)\to H(w) \quad \mbox{as}
\quad j\to \infty\quad\mbox{in}\quad L^1(B_1).
\ege
Note that the weak formulation of \eqref{eqn:r1} along
the sequence $\{t_j\}$ is given by
\bge\label{ps3}
\int_{B_1} \frac{\partial{u_j}}{\partial t}\,\phi \Ix =
- \int_{B_1} \nabla u_j\cdot\nabla \phi \Ix
+ \la H^{-1}(u_j)\int_{B_1} \phi (1-u_j)^{-2} \Ix
\quad \mbox{as} \quad j\to \infty
\ege
for any $\phi \in H_0^1(B_1).$

Passing to the limit as $j\to\infty$ in \eqref{ps3}, and
in conjunction with \eqref{tdc0}, \eqref{q8}, \eqref{q10}
and \eqref{ps2}, we derive
\bgee
\Delta u_j + \frac{\la}{(1-u_j)^2 \Big( 1 + \int_{\Omega} \frac{\dd x}{1-u_j}
\Big)^2} \to 0 \quad\mbox{in}\quad L^2(\Omega),
\egee
which implies that $w$ is an weak finite-energy solution of problem
\eqref{nss} corresponding to $\la>\la^*$, contradicting
the result of Proposition \ref{sp1}.

On the other hand, for $N=1$ by using \eqref{tdc0},
\eqref{eqn:r1}, (\ref{ps}), \eqref{q4} and \eqref{ps4} we deduce
that  $(u_j)_x$ is bounded in $W^{1,1}(-1,1)$ and thus,
by virtue of Sobolev's inequality, 
\bge\label{psh}
(u_j)_x\quad\mbox{is bounded in}\quad L^{\infty}(-1,1)
\ege

Furthermore
\bgee
\left[(1-u_j)^{-1}\right]\quad\mbox{is bounded in}\quad W^{1,1}(-1,1)
\egee
since
\bgee
\left[(1-u_j)^{-1}\right]_x=\frac{(u_j)_x}{(1-u_j)^2}\quad\mbox{is bounded in}\quad L^{1}(-1,1),
\egee
due to \eqref{q4} and \eqref{psh},  and
\bgee
(1-u_j)^{-1}\quad\mbox{is bounded in}\quad L^{1}(-1,1),
\egee
by virtue of \eqref{ps0}.

Therefore Sobolev's inequality guarantees that
\bgee
\left[(1-u_j)^{-1}\right]\quad\mbox{is bounded in}\quad L^{\infty}(-1,1)
\egee
and thus 
\bgee
\left[(1-u_j)^{-2}\right]\quad\mbox{is bounded in}\quad L^{\infty}(-1,1).
\egee
Now by virtue of \eqref{q9} we derive that
\bgee
(1-u_j)^{-1}\to (1-w)^{-1}\quad\mbox{and}\quad(1-u_j)^{-2}\to (1-w)^{-2}\quad \quad\mbox{as}\quad j\to \infty\quad\mbox{in}\quad L^{\infty}(-1,1).
\egee
Consequently,
\bgee
\Delta w + \frac{\la}{(1-w)^2 \Big( 1 + \int_{\Omega} \frac{\dd x}{1-w}
\Big)^2}=0 \quad\mbox{in}\quad L^2(\Omega),
\egee
where the non-local term is bounded and hence
elliptic regularity arguments entail that $w$ classical steady solution,
again contradicting Proposition~\ref{sp1}.
\end{proof}

\begin{rem}
Theorem \ref{que1} improves the results of Theorems 4.1  in \cite{GHW08}
and Theorem 5.2, 5.3 in \cite{H11}. Indeed, the earlier results
have provided finite-time quenching only for large values of the parameter
$\la$, without giving a threshold for $\la$ above which quenching occurs.
\end{rem}


\subsection{Quenching for large initial data}

To prove quenching for big initial data, i.e. for $0<u_0(x)<1$
close to $1$, we employ the widely used classical technique
of Kaplan, \cite{kaplan}. The estimate is provided by Lemma~\ref{est2},
permitting us to treat the non-local problem \eqref{eqn:r1}-\eqref{ric}
as a local one. In particular we have:
\begin{thm}
For any $\la>0$ there exist symmetric initial data $u_0$ satisfying the assumptions of Theorem \ref{que1} that are close to $1$ such that the solution $u$ of \eqref{eqn:r1}-\eqref{ric} quenches in finite time $T_q<\infty.$
\end{thm}
\begin{proof}
Set $\la_1=\la_1(B_1)>0$ the principal eigenvalue of the following problem
\bgee
-\Delta \phi=\la \phi,\;\; x\in B_1,\;\;\phi(x)=0,\;\; x\in \pl B_1 \, ,
\egee
with associated positive  eigenfunction $\phi_1(x)$ normalized so that
\bgee
\int_{B_1} \phi_1(x) \Ix = 1.
\egee
Let us assume that problem \eqref{eqn:r1}-\eqref{ric} has
a global-in-time solution, i.e. $T_{max}=\infty$ so that $0<u(x,t)<1$ for any
$(x,t)\in B_1\times(0,+\infty).$

Multiplying \eqref{eqn:r1} by $\phi_1,$ integrating over $B_1$
and using Green's second identity, we obtain, via Lemma \ref{est2},
\bge\label{al5}
\dif At & = & - \la_1 A(t)+\frac{\la \int_{B_1} \phi_1 (1-u)^{-2} \Ix
}{H^2(u)} \no \\
& \geq & - \la_1 A(t)+\frac{\la \int_{B_1} \phi_1 (1-u)^{-2} \Ix}{C_2^2}
\ege
where $A(t)=\int_{B_1} u\,\phi_1 \Ix.$
Applying Jensen's inequality to \eqref{al5},
\bge
\dif At \geq - \la_1 A(t) + \frac{\la}{C_2^2} (1-A(t))^{-2}
\quad\mbox{for any} \quad t>0 \, .
\ege
Choosing $\gamma \in (0,1)$ so that
$$
\Psi(s) := \frac{\la}{C_2^2} (1-s)^{-2} - \la_1 s > 0 \quad
\mbox{ for all $s \in [\gamma,1)$,}
$$
then by choosing $u_0$ close enough to $1$ such that
$A(0)\geq \gamma$, relation \eqref{al5} yields
\bgee
\dif At \geq \Psi(A(t)) > 0 \quad\mbox{for any}\quad t>0,
\egee
which then leads to
\bgee
t \leq \int_{A(0)}^{A(t)} \frac{\dd s}{\Psi(s)} \leq
\int_{A(0)}^{1} \frac{\dd s}{\Psi(s)}<\infty,
\egee
contradicting the  assumption $T_{max}=\infty$.
This completes the proof of the theorem.
\end{proof}


\section{Local Behaviour at Quenching}\label{spq}

\setcounter{equation}{0}


In this section we obtain some limited results about the manner
of finite-time quenching, shown to take place in the previous section.


\subsection{Single-point quenching}

Our main result is that for the class of problems under consideration,
namely radially symmetric with monotonic decreasing initial data,
quenching, when it occurs, takes place at a single point, the origin:
\begin{thm}
If we consider initial data as in Theorem \ref{que1} so that the solution of problem \eqref{eqn:r1}-\eqref{ric} quenches in finite time $T_q<\infty$, the quenching occurs only at the origin $r=0.$
\end{thm}
\begin{proof}
The proof follows immediately from from (\ref{ps0a}).
\end{proof}

Due to the non-locality, obtaining the sharp profile of
the standard (local) problem (cf. \cite{FG}) for \eqref{eqn:r1}-\eqref{ric} might be hard. However it can be rigorously
shown that the exponent $2/3$ in \eqref{ps0a} is optimal, at least in the sense that \eqref{ps0a} cannot be true
for any exponent $k < 2/3.$

The optimality of the exponent $2/3$ is a consequence
of the following result.
\begin{prop}\label{phs}
Let $T_q$ be the quenching time of the solution $u$ of \eqref{eqn:r1}-\eqref{ric} then
\bge\label{qbu}
\lim_{t\to T_q}||(1-u)^{-1}||_m=\infty\quad\mbox{for any}\quad m> \frac{3N}{2}>1.
\ege
\end{prop}
\begin{proof}
First note $\theta=(1-u)^{-1}$ satisfies $\theta_t-\Delta \theta
\leq f(t) \theta^4\leq \lambda \theta^4$ with $\theta=1$ on
$\pl \Omega.$ Next fix any $\Lambda>0$ and assume that $||\theta(t_0)||_m
\leq \Lambda$ for some $m> 3N/2>1$ and $t_0\in (0,T_q).$

By virtue of \cite[Theorem 15.2, Example 51.27]{QS07}, see also \cite[Theorem 1]{BC} and \cite[Theorem 1]{W1},  we have that problem
\bgee
&& z_t-\Delta z=\lambda (1+z)^4\quad \mbox{in}\quad \Omega\times (t_0,T_q)\\
&& z=0\quad \mbox{on}\quad \pl\Omega\times (t_0,T_q)\\
&&z(x,t_0)= \theta(x,t_0)\quad x\in \Omega
\egee
is well posed and in particular there exists $\tau>0$ such that
\bge\label{uest}
||z(t_0+s)||_{\infty}\leq K s^{-N/2m},\quad s\in(0,\tau],
\ege
where $K,\tau$ depend only on $\Lambda, m, \Omega, \lambda.$
By comparison $\theta$ exists and satisfies $\theta\leq z+1$
on $[t_0,t_0+\tau].$ Since $\lim_{t\to T_q}||\theta(t)||_{\infty}=\infty$
it follows that $||\theta(t)||_m>\Lambda$ for all
$t\in\left(\max(0,T_q-\tau),T_q\right)$ and thus \eqref{qbu}.
\end{proof}
\begin{rem}
It should be noted in the critical case $m = m_c = 3N/2$ that the time $\tau$
in \eqref{uest} depends on $\theta(t_0)$ and not just
on $\|\theta(t_0)\|_m$ (see \cite{BC} and
\cite[Remarks 15.4 and 16.2(iv)]{QS07}). Therefore, in this case,
it is no longer certain that $T_q<\infty$ implies
the finite-time blow-up of the norm $\|\theta(t_0)\|_m.$
\end{rem}
\begin{cor}\label{cor:bdexp}
Relation \eqref{ps0a} is not valid for any $k < 2/3.$
\end{cor}
\begin{proof}
First note that since $T_q = T_{max}<\infty$ it is easily seen
by the proof of Lemma \ref{est2} that \eqref{ps0a} is valid up to $T_{max}.$
Assume now that there is $k_0<2/3$ such that
\bgee
1-u(r,t)\geq C(k_0) r^{k_0}\quad\mbox{for}\quad (r,t)\in (0,1)\times (0,T_{max}]
\egee
then
\bgee
\lim_{t\to T_{max}}||(1-u)^{-1}||_m=\int_0^1 \frac{r^{N-1}}{(1-u(T_{max}))^{m}}\,dr\leq  C^{-1}(k_0) \int_0^1 r^{N-1-mk_0}\,dr<\infty
\egee
for $m > 3N/2$ close to $3N/2$, contradicting Proposition~\ref{phs}.
\end{proof}


\subsection{Lower bound of quenching rate}

We recall that considering radial decreasing initial data $u_0$ then $u$ inherits this property and hence
\bgee
M(t):= \max_{x\in\bar{B_1}} u(x,t)=u(0,t).
\egee
The next result provides a lower estimate of the quenching rate:
\begin{thm} \label{thm:bound}
The lower bound of the quenching rate of problem \eqref{eqn:r1}-\eqref{ric}
is given by
\bge\label{lb}
M(t)\geq 1-\widehat{C}(T_q-t)^{1/3}\quad\mbox{for}\quad 0<t<T_q \, ,
\ege
where $\widehat{C}$ is a positive constant independent of time $t.$
\end{thm}
\begin{proof}
It can be easily checked that the function $M(t)$ is Lipschitz continuous
and hence, by Rademacher's theorem, is almost everywhere differentiable,
see \cite{FM, KN07}. Furthermore, since $u$ is decreasing in $r$,
$\Delta_r u(0,t)\leq 0$ for all $t\in(0,T_q)$. Therefore, for any $t$
where $\dd M/\dd t$ exists, we have
\bgee
\dif Mt \leq \la \frac{(1-M(t))^{-2}}{\left( 1+ \int_{B_1}
\frac{1}{1-u} \Ix \right)^2} \leq \la \frac{(1-M(t))^{-2}}
{\left( 1 + N \omega_N \right)^2} \quad \mbox{for a.e.} \quad t\in(0,T_q),
\egee
which yields
\bgee
\int_{M(t)}^1 (1-s)^2 \I s \leq \la C (T_q-t),
\egee
for $C=1/(1 + N \omega_N)^2$, giving the desired estimate
\bgee
M(t) \geq 1 - \widehat{C}(T_q-t)^{1/3} \quad \mbox{for} \quad 0<t<T_q \, ,
\egee
where $\widehat{C}=(3\la C)^{1/3}.$
\end{proof}


\section{Numerical Results}\label{nr}

\setcounter{equation}{0}


We now carry out a brief numerical study of problem
(\ref{eqn:1})-(\ref{ic}) for the one-dimensional case.
Here the form of the problem is taken as

\begin{eqnarray}\label{pru}
&u_t = u_{xx} \displaystyle
+ \frac{\lambda}{\left(1-u \right)^2 \left( 1+
\int_{0}^{1}\frac{1}{1-u} \Ix\right)^2},
\quad 0<x<1,\quad t>0,\\
&u(0,t)=0,\quad u(1,t)=0,\nonumber\\
&u(x,0)=u_0(x).\nonumber
\end{eqnarray}

A moving mesh adaptive method, based on the techniques suggested in [4],
is used. This captures the behaviour of the solution near a singularity.
More specifically we take initially a partition of $M+1$ points in $[0,1]$,
$\xi_0 = 0, \, \xi_0 + \delta \xi = \xi_1,\cdots, \, \xi_M=1$.
For the solution $u=u(x,t),$ we introduce a computational coordinate
$\xi$ in the interval $[0,1]$ and
we consider the mesh points $X_i$ to be the images of the points
$\xi_i$ (uniform mesh) under the map $x(\xi,t)$ so that
$X_i(t)= x(i\delta \xi,t)$.
Given this transformation, we have, for the approximation of the solution
$u_i(t)\simeq u(X_i(t),t),$ that
$\displaystyle \dif{u(X_i(t),t)}t = u_t(X_i,t)+u_x \dot X_i$ or
$\displaystyle u_t= \dif ut -u_x x_t$.

The way that the map, $x(\xi,t),$ is determined is controlled
by the monitor function ${\mathcal M}(u)$ which, in a sense,
follows the evolution of the singularity. This function is
determined by the scale invariants of the problem, \cite{Budd}.
In our case, for the semilinear parabolic equation of the form
$v_{t}=v_{xx} - \lambda/(v^2(1+\int_{0}^{1} v^{-2} \Ix )^2)$
for $v=1-u$, an appropriate monitor function should be
${\mathcal M}(v)=|v|^{-2}$ in this case.

At the same time we need also a rescaling of time  of the form
$\displaystyle \dif ut = \dif u\tau \dif\tau t$
for $\displaystyle \dif t\tau =g(u)$,
where $g(u)$ is a function determining the way that the time scale
changes as the solution approaches the singularity,
and is given by $g(u)= 1/\|{\mathcal M}(u)\|_{\infty}$ (again see \cite{Budd}).

In addition the evolution of $X_i(t)$ is given by a moving mesh PDE 
which has the form $-x_{\tau\xi\xi} = \varepsilon^{-1} g(u)
\left({\mathcal M}(u)x_{\xi} \right)_{\xi}$. Here $\varepsilon$
is a small parameter accounting for the relaxation time scale.

Thus finally we obtain a system of ODE's for $X_i$ and $u_i$
and the ODE system takes the form
\begin{eqnarray*}
\dif t\tau & = & g(u),\\
u_{\tau}-x_{\tau}u_x & = &
g(u) \left(u_{xx} +
\frac{\lambda}{\left(1-u \right)^2\left( 1 + \int_{0}^{1}
\frac{1}{1-u} \Ix \right)^2} \right),\\
- x_{\tau\xi\xi} &=& \frac{g(u)}{\varepsilon}
\left( {\mathcal M}(u)x_{\xi}\right)_{\xi}.
\end{eqnarray*}

We may apply now a discretization in space and we have
 \begin{eqnarray*}
u_x(X_i,\tau) &\simeq & \delta_x u_i(\tau):= \frac{u_{i+1}(\tau) - u_{i-1}(\tau) }{X_{i+1}(\tau) - X_{i-1}(\tau)},\\
u_{xx}(X_i,\tau) &\simeq & \delta_x^2 u_i(\tau):=
\left( \frac{u_{i+1}(\tau) - u_{i}(\tau) }{X_{i+1}(\tau) - X_{i}(\tau)}-
 \frac{u_{i}(\tau) - u_{i-1}(\tau) }{X_{i}(\tau) - X_{i-1}(\tau)}\right)\frac{2} {X_{i+1}(\tau) - X_{i-1}(\tau)},\\
x_{\xi\xi}(\xi_i,\tau) &\simeq &  \delta_{\xi}^2 x_i(\tau):= \frac{X_{i+1}(\tau) - 2X_{i}(\tau) +X_{i-1}(\tau)}{\delta\xi^2},\\
 \left( {\mathcal M}(u)x_{\xi}\right)_{\xi} &\simeq &  \delta_{\xi} ({\mathcal M}\delta_{\xi} x):=
 \left(\frac{{\mathcal M}_{i+1}-{\mathcal M}_{i}}{2}\frac{x_{i+1}-x_{i}}{\delta\xi} -
\frac{{\mathcal M}_{i}-{\mathcal M}_{i-1}}{2}\frac{x_{i}-x_{i-1}}{\delta\xi}  \right)\frac{1}{\delta\xi}.
\end{eqnarray*}

Therefore the resulting ODE system to be solved, for
\begin{eqnarray*}
y&=&\left(t(\tau),  u_1(\tau), u_2(\tau),\ldots  u_M(\tau), X_1(\tau), X_2(\tau),\ldots  X_M(\tau) \right),\\
&=&\left(t(\tau), {\bf u}, {\bf X}\right),\quad\quad   \,{\bf u},\, {\bf X}\in\mathbb{R}^M,
\end{eqnarray*}
will have the form
  \begin{eqnarray*}
A(\tau,y )\dif y\tau = b(\tau,y),
\end{eqnarray*}
where the matrix $A\in\mathbb{R}^{2n+1}$ has the block form
   \begin{eqnarray*}
A=\left[ \begin{array}{ccc}
 1 & 0 &  0\\
  0 & I  & -\delta_x u\\
  0 & 0  & -\delta_{\xi}^2
\end{array}
\right],\quad
y=\left[ \begin{array}{c}
 t(\tau)\\
 {\bf u}\\
  {\bf X}
\end{array}
\right],
\quad
b=g(u)\left[ \begin{array}{c}
1\\
\delta_x^2 {\bf u} + \lambda \frac{1}{\left(1-{\bf u}\right)^2
\left(1+{\mathcal I}(u)\right)^2}\\
\delta_{\xi} ({\mathcal M}\delta x_{\xi})
\end{array}
\right].
\end{eqnarray*}
where ${\mathcal I}(u)$ is an approximation of the integral
$\int_{0}^{1}\frac{1}{1-u} \Ix$, using, for example, Simpson's rule.

For the solution of the above system a standard ODE solver,
such as the matlab function ``ode15i", can be used.

\

We first plot, in Figure~\ref{fig1}, the solution for $\lambda>\lambda^*=8.533$, \cite{KLNT11}, to give quenching.
The initial data are taken to be zero in this case; we use $M=141$.
In can be observed that the solutions flattens for a time (this will
be while it lies close to the steady state corresponding to
$\la = \la^*$).
\begin{figure}[h]
\input{epsf}
\begin{center}
\epsfysize=6cm \epsfxsize=10cm \epsfbox{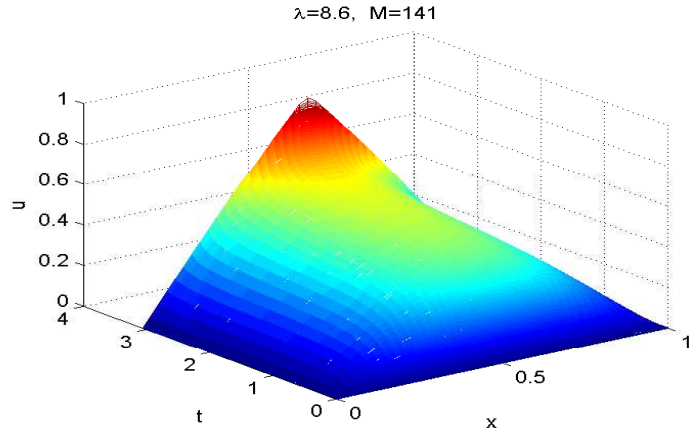}
\vspace*{.6cm}
\end{center}
\caption{ \it The numerical solution of problem (\ref{pru})
with $u_0 \equiv 0$, for $\lambda=8.6$ and taking $M=141$. \label{fig1}}
\end{figure}

Using now $\lambda=10$, we can see the evolution of the solution profile
against space for various times in Figure~\ref{fig:prof}. Again we take
$u_0 \equiv 0$ and use $M=141$.
\begin{figure}[h]
\input{epsf}
\begin{center}
\epsfysize=6cm \epsfxsize=10cm \epsfbox{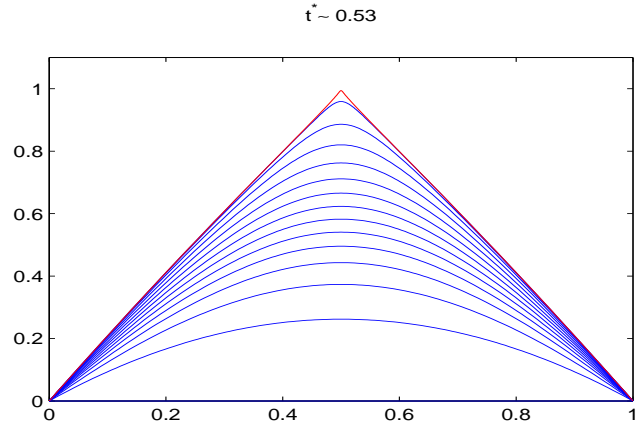}
\vspace*{.6cm}
\end{center}
\caption{ \it  Profile of the  numerical solution of problem
(\ref{pru}) against $x$ for $\la = 10$ and $u_0 \equiv 0$,
taking $M=141$. \label{fig:prof}}
\end{figure}

Regarding the behaviour of the solution near quenching with respect to time,
our numerical simulations indicate that we get an approximate $t^{1/3}$
dependence. More precisely near the quenching time $T=T_q$,
$\ln \left( 1 -u(\frac12,t)
\right) \propto \ln\left(T-t\right)$ with constant
of proportionality $\frac13$. This is demonstrated in
Figure~\ref{figq3}. For the local spatial dependence, a plot of $\ln u(x,T)$
against  $\ln (x - \frac 12)$, in
Figure~\ref{figq4}, shows that $u(x,T)$ behaves approximately like
$\sim C(x - \frac12)^{\frac23}$ near quenching.
These numerical results agree both with the bound of Theorem~\ref{thm:bound}
and the asymptotic results on quenching of
\cite{FG, GHW08, G08, YG-ZP-MJW06}. (Note that a more accurate
local asymptotic form of the quenching profile is expected to be, \cite{FG},
$u \sim C(x - \frac12)^{\frac23} |\ln (x - \frac 12|)|^{- \frac 13}$.
Note also that the local quenching behaviour of our non-local
problem is expected to be like that of the standard problem, (\ref{local}),
because of the boundedness of the integral in the non-local term,
Lemma~\ref{est2}.)

\begin{figure}[h]
\input{epsf}
\begin{center}
\epsfysize=6cm \epsfxsize=10cm \epsfbox{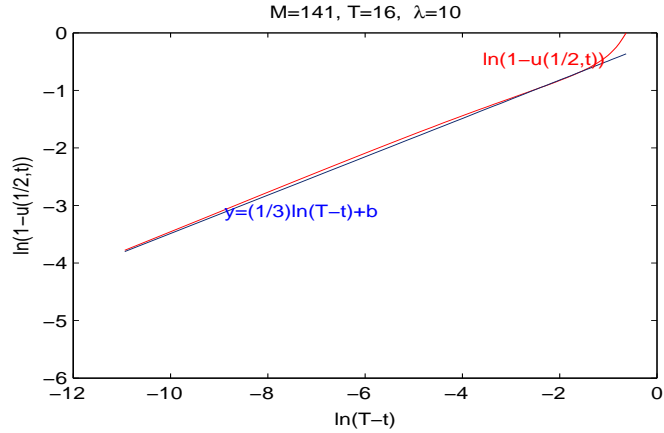}
\vspace*{.6cm}
\end{center}
\caption{ \it Plot of $y = \ln\left(1-u(\frac12,t)\right)$
(red curve) against $ \ln\left(T - t\right)$ for $\lambda=10$.
The straight line (blue) has slope $\frac13$
and indicates good agreement between $1-u(\frac12,t)$
and const.$\times (T-t)^{\frac13}$.  }\label{figq3}
\end{figure}


\begin{figure}[h]
\input{epsf}
\begin{center}
\epsfysize=6cm \epsfxsize=12cm \epsfbox{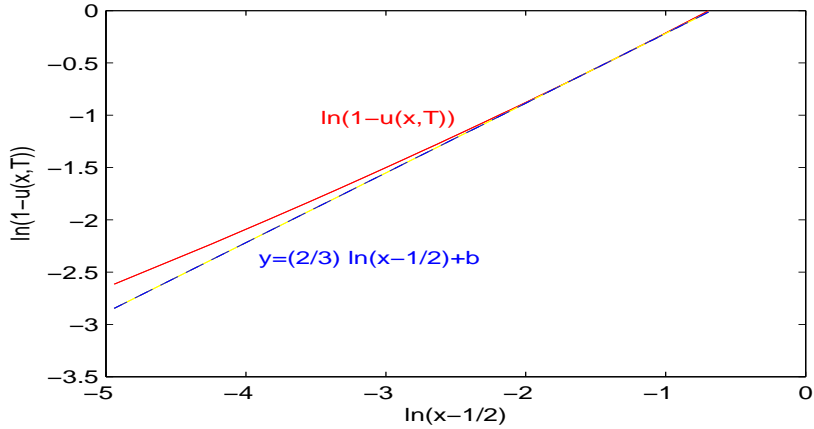} \vspace*{.6cm}
\end{center}
\caption{ \it { Plot of $\ln(1-u(x,T))$ (solid curve) against
$\ln (x - \frac 12)$, for $\lambda=10$.
The blue dashed line has slope $2/3$. }}
\label{figq4}
\end{figure}



We look briefly at the radial symmetric problem in two dimensions,
\begin{eqnarray}
u_t & = &u_{rr} + \frac{1}{r}u_r+
\frac{\lambda}{\left(1-u \right)^2\left(1+\, 4\pi\int_{0}^{1}
\frac{r}{1-u} \I r\right)^2},
\quad 0<r<1,\quad t>0 \, , \nonumber \\
 u_r(0,t) & = & 0,\quad u(1,t)=0 \, , \label{pro} \\
 u(r,0) & = & u_0(r) \, , \nonumber
\end{eqnarray}
taking, for a change, $\alpha = 2$.

We plot the solution for $\lambda=71$ in Figure~\ref{fig2D}.
Figure \ref{radprof} shows the profile of the solution for various times.
We again take $u_0 \equiv 0$ and use $M=141$.
We see that the behaviour is very similar to the one-dimensional problem.
Temporary flattening can again be observed.
\begin{figure}[h]
\input{epsf}
\begin{center}
\epsfysize=6cm \epsfxsize=10cm \epsfbox{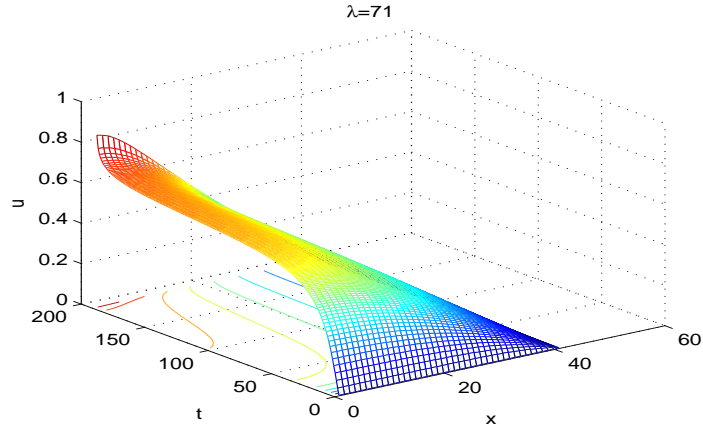}
\vspace*{.6cm}
\end{center}
\caption{ \it The numerical solution of problem (\ref{pro}),
for $\lambda=71$ with $u_0 \equiv 0$ using $M=141$. \label{fig2D}}
\end{figure}

\begin{figure}[h]
\input{epsf}
\begin{center}
\epsfysize=6cm \epsfxsize=10cm \epsfbox{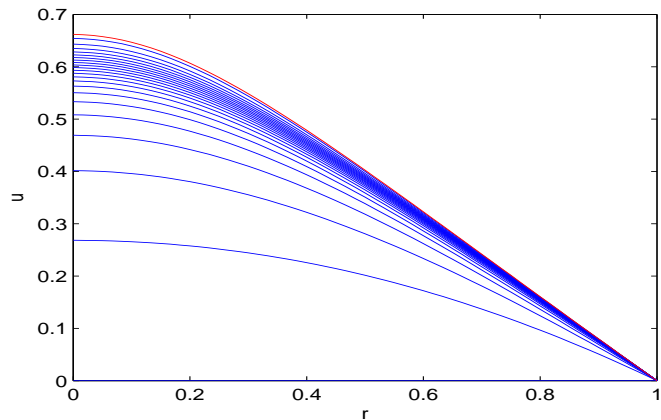}
\vspace*{.6cm}
\end{center}
\caption{ \it Profiles for various times of the numerical solution of problem (\ref{pro}),
for $\lambda=71$ with $u_0 \equiv 0$ using $M=141$. \label{radprof}}
\end{figure}


\section{Discussion}\label{sec:discuss}

Our results clearly easily extend to other problems of a similar
form. For example, the analysis holds
with $F(r,t)$, the right-hand side in (\ref{eqn:r1}),
again given by $F(r,t) = \la  k(t) (1 - u(r,t))^{-2}$ now with $k(t) =
\calF (\int_\Omega (1-u)^{-1} \Ix )$ and $\calF (s)$ positive
and satisfying $\calF (s) \gg s^{-(a+1)}$ as $s \to \infty$ for some
$a \in (1,2)$. Other non-linear functions of $u$ multiplying $k(t)$
are also possible.

Local behaviour which has been established for the ``standard''
problem should be expected to carry over to these non-local
problems on account of the boundedness of $\int_\Omega (1-u)^{-1} \Ix$.

We also expect the results to hold for non-monotone initial data in the unit
ball and for asymmetric problems in more general domains $\Omega$.

\setcounter{equation}{0}
\subsection*{Acknowledgement}The authors would like to thank
the anonymous referee for their stimulating comments.  In particular,
their suggestions and recommendations regarding Theorem~\ref{que1},
Proposition~\ref{phs} and Corollary~\ref{cor:bdexp} significantly
improved the paper.




\end{document}